\documentclass[11pt]{amsart}
\usepackage{amsmath,amssymb,amsthm}
\usepackage{graphicx}
\usepackage{hyperref}
\usepackage[margin=1in]{geometry}
\usepackage{cite}

\theoremstyle{plain}
\newtheorem{theorem}{Theorem}[section]
\newtheorem{lemma}[theorem]{Lemma}
\newtheorem{proposition}[theorem]{Proposition}

\theoremstyle{definition}
\newtheorem{definition}[theorem]{Definition}

\newtheorem{remark}[theorem]{Remark}

\numberwithin{equation}{section}

\title{A Generalisation on Erd\H{o}s Distinct Subset Sums Problem}
\author{Zijie Gu}

\begin{document}
\maketitle

\begin{abstract}
This paper investigates the Erd\H{o}s distinct subset sums problem in $\mathbb{Z}^k$. Beyond the classical variance method, using alternative statistical quantities like $\mathbb{E}[\|X\|_1]$ and $\mathbb{E}[\|X\|_3^3]$ can yield better bounds in certain dimensions. This innovation improves previous low-dimensional results and provides a framework for choosing suitable methods depending on the dimension. 
\end{abstract}

\section{Introduction}

The Erd\H{o}s distinct subset sums problem asks for the minimum value of the largest element in a sequence of positive integers such that all subset sums are distinct. Formally, given $A = \{a_1, \ldots, a_n\} \subset \mathbb{N}$ with $a_1 < a_2 < \ldots < a_n$, if all $2^n$ subset sums are distinct, what is the minimum possible value of $a_n$? This problem has been studied extensively over the years by many researchers \cite{aliev2008siegel, alon2016probabilistic, bae1996subset, elkies1986improved, erdos1941problems, guy1982sets,   steinerberger2023some,dubroff2021note}.

Erd\H{o}s and Moser \cite{erdos1941problems} established the foundational bound
\[
a_n \geq \frac{1}{4} \cdot \frac{2^n}{\sqrt{n}}
\]
using the variance method\footnote{This method computes the variance of a statistical variable generated by $a_1,a_2,\ldots,a_n$ from two perspectives to obtain a lower bound for $a_n$.}. This bound has been progressively improved, culminating in the current best result by Dubroff, Fox, and Xu \cite{dubroff2021note}:
\[
a_n \geq \sqrt{\frac{2}{\pi}} \cdot \frac{2^n}{\sqrt{n}}.
\]

\begin{definition}
    Let $2^{[n]}$ denote the family of all subsets of $\{1,\ldots,n\}$. Define $M$ to be the minimum integer such that there exists a sequence $\Sigma = (a_1,\ldots,a_n)$ in $\mathbb{Z}^k$ with each $a_i \in [0,M]^k$ (i.e. $\Sigma$ is $M$-bounded) and
    \[
    \sum_{i \in A_1} a_i \neq \sum_{i \in A_2} a_i
    \]
    for all distinct $A_1, A_2 \in 2^{[n]}$. 
    
    We call such a sequence a \emph{$M$-bounded distinct subset sum sequence} in $\mathbb{Z}^k$.
\end{definition}

This high-dimensional generalisation, where integers are replaced by vectors in $\mathbb{Z}^k$, has only received limited attention. Costa, Dalai, and Della Fiore \cite{costa2021variations} extended the variance method to obtain:
\[
M \geq (1+o(1)) \cdot \sqrt{\frac{4}{\pi n(k+2)}} \cdot \Gamma\left(\frac{k}{2}+1\right)^{1/k} \cdot 2^{n/k}
\]
where $M$ denotes the maximum component value in the vectors.

Our main result gives an improved lower bound for $M$ in several dimensions $k$. In particular, by considering more general statistical quantities beyond the classical variance, we obtain improved bounds for certain $k$. After a careful comparison between my result and Costa, Dalai, and Della Fiore's, the main theorem of this paper is as follows:
\begin{theorem}\label{mainthe}
    Let $\Sigma = (a_1,\ldots,a_n)$ be a \emph{$M$-bounded distinct subset sum sequence} in $\mathbb{Z}^k$. Then:
\[
    M \geq 
    \begin{cases}
        (1+o(1))\;\cdot\;\frac{1}{k+1}\;\cdot\;\sqrt{\frac{\pi}{2}}\;\cdot\;\frac{2^{n/k}}{\sqrt{n}}\;\cdot\;(k!)^{1/k} & k \leq 4 \\[1.5ex]
        (1+o(1))\;
        \cdot\;
        \frac{1}{\sqrt[3]{k+3}}\;
        \cdot\;
        \sqrt[6]{\frac{\pi}{8}}\;
        \cdot\;
        \frac{2^{n/k}}{\sqrt{n}}\;
        \cdot\;
        \frac{\Gamma\!\big(\tfrac{k+3}{3}\big)^{1/k}}{\Gamma\!\big(\tfrac{4}{3}\big)} & 4 < k \leq 6 \\[1.5ex]
        (1+o(1))\;\cdot\;\frac{1}{\sqrt{k+2}}\;\cdot\;\sqrt{\tfrac{4}{\pi}}\;\cdot\;\frac{2^{n/k}}{\sqrt{n}}\;\cdot\;\Gamma\!\big(\tfrac{k}{2}+1\big)^{1/k} & k > 6
    \end{cases}
\]
\end{theorem}
\section{Conversion of the Distinct Subset Sum}
To analyse the distinctness of subset sums, we introduce the random variable
\[
X = \sum_{i=1}^{n} \epsilon_i a_i,
\]
where each $\epsilon_i$ is an independent random variable that takes the value $-\frac{1}{2}$ or $\frac{1}{2}$ with equal probability. 

This construction is for any subset $A \subseteq [n]$, consider
\[
\sum_{i \in A} a_i - \frac{1}{2} \sum_{i=1}^n a_i = \sum_{i=1}^n \epsilon_i a_i,
\]
where each $\epsilon_i = \frac{1}{2}$ if $i \in A$ and $\epsilon_i = -\frac{1}{2}$ if $i \notin A$. Therefore, $X$ should also exhibit the property of being distinct for different subsets $A$.
\section{Statistical Bridge Methods}

\subsection{General Framework}

The core idea is to bound a statistical quantity (hereafter, we refer to such a quantity as a \emph{Statistical Bridge}) from above and below:
\begin{itemize}
\item \textbf{Upper bound:} Obtained when all vector components achieve maximum values
\item \textbf{Lower bound:} Derived from the $2^n$ points closest to the origin
\end{itemize}

\subsection{Volume Calculations in $p$-Normed Spaces}

\begin{definition}
Let $1 \leq p < \infty$ and $k \in \mathbb{N}$. The \emph{$p$-norm} of a vector $X = (x_1, x_2, \ldots, x_k) \in \mathbb{R}^k$ is defined by
\[
    \|X\|_p = \left( \sum_{i=1}^k |x_i|^p \right)^{1/p}.
\]
A \emph{$k$-dimensional $p$-normed space} is the metric space $(\mathbb{R}^k, d_p)$, where the distance between $X, Y \in \mathbb{R}^k$ is given by
\[
    d_p(X, Y) = \|X - Y\|_p.
\]
\end{definition}
   
   \begin{proposition}\label{thm:pnorm-volume}
   The volume $V_{k,p}(R)$ of a $k$-dimensional ball with radius $R$ in the $p$-norm, that is,
   \[
   B_{k,p}(R) = \left\{ X \in \mathbb{R}^k : \|X\|_p \leq R \right\},
   \]
   is given by
   \[
   V_{k,p}(R) = \frac{\left[2\Gamma\left(1+\frac{1}{p}\right)\right]^k}{\Gamma\left(1+\frac{k}{p}\right)} \cdot R^k,
   \]
   where $\Gamma$ denotes the Gamma function.
   \end{proposition}

   \begin{proposition}\label{thm:pnorm-surface}
   The surface area $S_{k,p}(R)$ of a $k$-dimensional $p$-norm sphere of radius $R$, i.e. the boundary of $B_{k,p}(R)$, is given by
   \[
   S_{k,p}(R) = \frac{d}{dR} V_{k,p}(R) = \frac{k \left[2\Gamma\left(1+\frac{1}{p}\right)\right]^k}{\Gamma\left(1+\frac{k}{p}\right)} \cdot R^{k-1}.
   \]
   \end{proposition}
   
   \begin{lemma}\label{lem:pnorm}
   In a $p$-normed space, let $A$ denote the set of $2^n$ integer lattice points (i.e., points in $\mathbb{Z}^k$) that are closest to the origin. Let $R$ be the radius of the smallest ball (centred at the origin) that contains all points in $A$. Then the following identity holds:
   \[
   \sum_{s\in A}{\|s\|_p^p}=2^n\cdot \frac{k}{k+p}\cdot R^p
   \]
   \end{lemma}

   \begin{proof}
   Using Proposition~\ref{thm:pnorm-volume}, we have:
   \[
   2^n = \frac{\left[2\Gamma\left(1+\frac{1}{p}\right)\right]^k}{\Gamma\left(1+\frac{k}{p}\right)} \cdot R^{k}
   \]
   Then use the Proposition~\ref{thm:pnorm-surface}, we have:
   \begin{align*}
   \sum_{s\in A}{\|s\|_p^p} &= R^p\cdot \sum_{s'\in A'}{\|s'\|_p^p} \\
   &= R^{p+k}\cdot\int_{\|s'\|\le1}\|s'\|_p^p\,ds' \\
   &= R^{p+k}\cdot \int_{0}^{1}\rho^p\cdot k\cdot\frac{\left[2\Gamma\left(\frac{p+1}{p}\right)\right]^k}{\Gamma\left( \frac{k+p}{p}\right)}\cdot\rho^{k-1}d\rho \\
   &= \frac{k}{p+k}\cdot R^{p+k}\cdot \frac{\left[2\Gamma\left(\frac{p+1}{p}\right)\right]^k}{\Gamma\left( \frac{k+p}{p}\right)} \\
   &= 2^n\cdot \frac{k}{p+k}\cdot R^p \qedhere
   \end{align*}
   \end{proof}

\section{Main Results}

\subsection{First Moment Method}
Using $\mathbb{E}[\|X\|_1]$ as the statistical bridge, we obtain the following theorem:
\begin{theorem}\label{thm:first}
We have:
\[
M \;\ge\; (1+o(1))\;\cdot\;\frac{1}{k+1}\;\cdot\;\sqrt{\frac{\pi}{2}}\;\cdot\;\frac{2^{\,\tfrac{n}{k}}}{\sqrt{n}}\;\cdot\;(k!)^{1/k}.
\]
\end{theorem}

\begin{proof}
Since the $k$ components of the vector $X$ are independent of each other, we have:
\[
\mathbb{E}[\|X\|_1] \leq k \cdot \max \mathbb{E}\left[\left|\sum_{i=1}^n \epsilon_i x_i\right|\right],
\]
where each $x_i$ is between $0$ and $M$.

The function $\mathbb{E}\left[\left|\sum_{i=1}^n \epsilon_i x_i\right|\right]$ is convex with respect to each variable $x_i$. Therefore, its maximum is achieved when each $x_i$ takes either the value $0$ or $M$. Let $n_0$ denote the number of $x_i$ equal to $M$, and the remaining $x_i$ are $0$. In this case, we have:
\begin{align*}
    \mathbb{E}\left[\left|\sum_{i=1}^n \epsilon_i x_i\right|\right] &= \frac{1}{2^n} \sum_{\epsilon_i} \left|\sum_{i=1}^n \epsilon_i x_i\right| \\
    &\leq \frac{1}{2^n} \cdot 2^{n-n_0} \cdot M \cdot \sum_{\epsilon_i} \left|\sum_{i=1}^{n_0} \epsilon_i\right| \\
    &= \frac{M}{2^{n_0}} \sum_{i=0}^{n_0} \binom{n_0}{i} \left| \frac{n_0}{2} - i \right| \\
    &= \frac{M}{2^{n_0}} \cdot n_0 \cdot \binom{n_0-1}{\left\lfloor \frac{n_0-1}{2} \right\rfloor} \\
    &\leq \frac{M}{2^n} \cdot n \cdot \binom{n-1}{\left\lfloor \frac{n-1}{2} \right\rfloor}
\end{align*}

Therefore,
\begin{equation}
\mathbb{E}[|X|]\le \frac{kMn}{2^n}\cdot \binom{n-1}{\lfloor\frac{n-1}{2}\rfloor}
\end{equation}

On the other hand, Lemma~\ref{lem:pnorm} (with $p=1$) gives a lower bound:
\begin{equation}
\mathbb{E}[|X|] \geq (1+o(1)) \frac{k}{k+1} R,
\end{equation}
where (using Proposition~\ref{thm:pnorm-volume})
\[
R=2^{n/k}\cdot \frac{\Gamma \left(\frac{k+1}{k}\right)^{1/k}}{2\Gamma (2)}=2^{n/k-1}\cdot \Gamma \left(\frac{k+1}{1}\right)^{1/k}=2^{n/k-1}\cdot \left(k!\right)^{1/k}
\]
Comparing these two bounds and solving for $M$ gives the result in the theorem~\ref{thm:first}.
\end{proof}

\subsection{Third Moment Method}
Using $\mathbb{E}[\|X\|_3^3]$ as the statistical bridge, we obtain the following theorem:
\begin{theorem}\label{thm:third}
We have:
\[
M \;\ge\; (1+o(1))\;
\cdot\;\frac{1}{\sqrt[3]{k+3}}
\;\cdot\;\sqrt[6]{\frac{\pi}{8}}
\;\cdot\;\frac{2^{\,\tfrac{n}{k}}}{\sqrt{n}}
\;\cdot\;\frac{\Gamma\!\big(\tfrac{k+3}{3}\big)^{1/k}}{\Gamma\!\big(\tfrac{4}{3}\big)}.
\]
\end{theorem}

\begin{proof}
Analogous to Theorem~\ref{thm:first}, the upper bound of $\mathbb{E}[\|X\|_3^3]$ is attained when all components of every variable achieve their maximum values. Consequently,
\begin{align*}
    \mathbb{E}[\|X\|_3^3] &\leq k \cdot \max \mathbb{E}\left[\left|\sum_{i=1}^n \epsilon_i x_i\right|^3\right] \\
    &\leq \frac{k}{2^n} \cdot M^3 \cdot \sum_{i=0}^n \binom{n}{i} \left|\frac{n}{2} - i\right|^3 \\
    &= 
    \left\{
    \begin{aligned}
        &\frac{kM^3}{2^n}\cdot \frac{n!}{\left((\frac{n}{2}-1)!\right)^2}, && n\ \text{even} \\[1.2ex]
        &\frac{kM^3}{2^n}\cdot \frac{n!\cdot(2n-1)}{4\left((\frac{n-1}{2})!\right)^2}, && n\ \text{odd}
    \end{aligned}
    \right.\\
    &\leq \frac{kM^3}{2^n}\cdot \frac{n!}{\left((\frac{n}{2}-1)!\right)^2}
\end{align*}

On the other hand, Lemma~\ref{lem:pnorm} (with $p=3$) gives a lower bound:
\begin{equation}
\mathbb{E}[\|X\|_3^3] \geq (1+o(1)) \cdot \frac{k}{k+3} \cdot R^3,
\end{equation}
where (using Proposition~\ref{thm:pnorm-volume})
\[
R = 2^{n/k-1} \cdot \frac{\Gamma\left(\frac{k+3}{3}\right)^{1/k}}{\Gamma\left(\frac{4}{3}\right)}.
\]

Comparing these two bounds and solving for $M$ gives the result in the theorem~\ref{thm:third}.
\end{proof}

\subsection{Dimensional Optimality}

In summary, by combining Theorems~\ref{thm:first} ,~\ref{thm:third} and Simone Costa, Marco Dalai, and Stefano Della Fiore's work~\cite{costa2021variations}, we see that the best lower bound for $M$ depends on the dimension $k$.

\begin{remark}[Theorem~\ref{mainthe}]
The optimal statistical method for bounding $M$ depends on the dimension $k$:
\[
M \geq \begin{cases}
\text{First moment bound (Theorem~\ref{thm:first})}, & k \leq 4 \\
\text{Third moment bound (Theorem~\ref{thm:third})}, & 4 < k \leq 6 \\
\text{Variance bound \cite{costa2021variations}}, & k > 6
\end{cases}
\]
\end{remark}

\begin{figure}[h]
    \centering
    \includegraphics[width=0.7\linewidth]{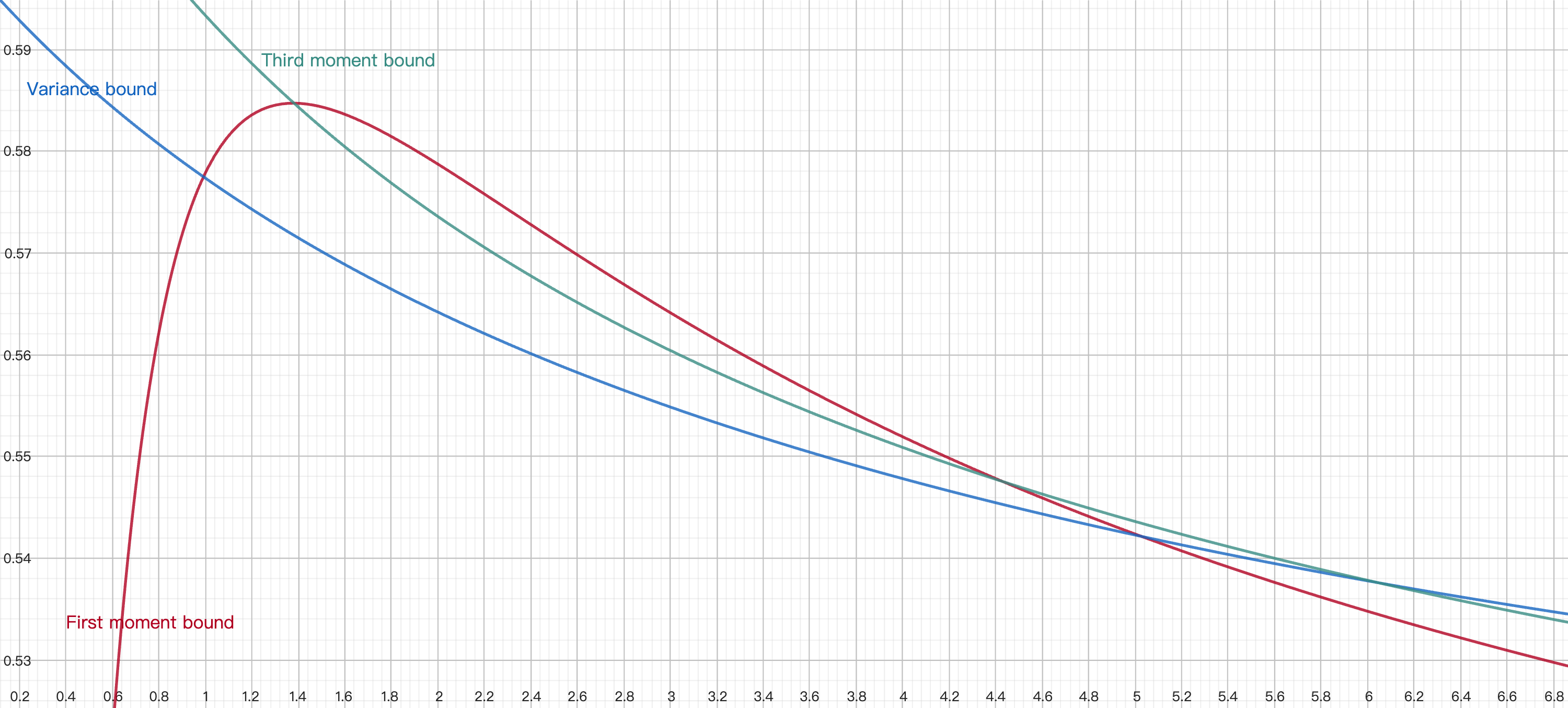}
    \caption{Comparison of coefficients relative to $k$ for three distinct methods}
    \label{coeff-comparison-relative-k}
\end{figure}

\section{Open Problems and Future Directions}

Several questions remain open:

\begin{enumerate}
\item \textbf{Higher-order statistics:} Can $\mathbb{E}[\|X\|_p^p]$ for $p > 3$ yield better bounds in some dimensional range?

\item \textbf{General statistics:} Can $\mathbb{E}[|X|_p^q]$ for some $p,q\in \mathbb{R}^+$ yield a best bounds for each dimensional range?

\item \textbf{Tight bounds:} What is the exact asymptotic behaviour of $M$ as a function of $n$ and $k$? This also encompasses Erd\H{o}s's conjecture.

\end{enumerate}

\section{Comment}

In this work, we have shown that although the variance method is a fundamental tool for the high-dimensional Erd\H{o}s problem, it is not always the most effective approach in every setting. By exploring alternative statistical bridges, we are able to obtain better bounds in lower dimensions and provide a framework for choosing suitable methods depending on the dimension. 

\bibliographystyle{plain}
\bibliography{dissertation.bib}

\end{document}